\documentclass[12 pt]{article}%
\usepackage{amsmath, amsfonts, amsthm, color,latexsym}
\usepackage{amsmath}
\usepackage{amsfonts}
\usepackage{amssymb}
\usepackage{color, soul}
\usepackage{graphicx}%
\usepackage{enumerate}
\usepackage[utf8]{inputenc}
\usepackage[T1]{fontenc}
\allowdisplaybreaks[4]
\providecommand{\U}[1]{\protect\rule{.1in}{.1in}}
\newtheorem{theorem}{Theorem}[section]
\newtheorem{proposition}[theorem]{Proposition}

\newtheorem{example}[theorem]{Example}

\newtheorem{remark}[theorem]{Remark}

\newtheorem{lemma}[theorem]{Lemma}
\newtheorem{final remark}[theorem]{Final Remark}
\newtheorem{definition}[theorem]{Definition}
\begin{document}

\title{\sc Multiple almost summing operators} 
\author{Joilson Ribeiro\thanks{joilsonor@ufba.br}~ and Fabr\'icio Santos\thanks{fabriciosantos@ufba.br\thinspace \hfill\newline\indent2010 Mathematics Subject
Classification: 46B45, 47L22, 46G25.\newline\indent Key words: Banach sequence spaces, ideals of homogeneous polynomials, linear stability, finitely determined.}}
\date{}
\maketitle

\begin{abstract} 
In this paper, we explore the concept of multilinear operators that are multiple almost summing and present a new concept of type and cotype of multilinear operators and investigate the conditions for this new concept to recover the original concept of type and type for multilinear operators. We also show that these classes are Banach multi-ideals and establish the coherence and compatibility of these classes with the generated homogeneous polynomials. 

\end{abstract}

\section{Introduction and background}


The notion of multiple summing operators was introduced, independently, in \cite{BG04, M03}, which is based on the successful theory of absolutely summing operators. 


The concept of multiple almost summing operators was initially introduced in \cite{P06} and its main motivation is the theory of almost summing multilinear operators (for more datails, see \cite{P04, PR12}). We explore this concept, showing that this class is a Banach multi-ideal and the sequence of pairs formed by this class and the class of the homogeneous polynomials generated are coherent and compatible with the class of almost summing linear operators. This class has presented itself as a challenge for the authors, because the sequence classes and $m$-sequence classes involved have not been finitely determined, these concepts having been introduced by Botelho and Campos in \cite{BC17} and by Ribeiro and Santos in \cite{RS19}. Among the difficulties encountered, we can highlight the characterization by inequalities so common in the study of classes of multilinear operators.
We also introduce a new concept of type and cotype of multilinear operator, which we call multi-type and multi-cotype, and we show that all multiple almost summing operators have some proper multi-type.

From now on, the letters $E,E_{1},\dots,E_{m},F,G,H$ will represent Banach spaces over the same scalar-field $\mathbb{K}=\mathbb{R}$ or $\mathbb{C}${\bf,} and $E'$ stands for the topological dual of $E$. The closed unit ball of $E$ is denoted by $B_E$. We use BAN to denote the class of all Banach spaces over $\mathbb{K}$. Given Banach spaces $E$ and $F$, the symbol $E\overset{1}\hookrightarrow F$ means that $E$ is a linear subspace of $F$ and $\Vert x\Vert_{F}\leq \Vert x\Vert_{E}$ for every $x \in E$. By $c_{00}(E; \mathbb{N}^m)$ we denote the set of all $E$-valued finite $m$-sequences, which, as usual, can be regarded as infinite $m$-sequences by completing with zeros. For every $k_1,\dots, k_m\in\mathbb{N}$, $e_{k_1,\dots, k_m} = (x_{j_1,\dots, j_m})_{j_1,\dots, j_m=1}^{\infty}$ is an $m$-sequence defined by:
\begin{equation*}
	x_{j_1,\dots,j_m} = \left\{\begin{array}{rc}
		1, &\text{if} \quad j_1=k_1,\dots, j_m=k_m\\
		0, &\text{otherwise}.
	\end{array}\right.
\end{equation*}
The space of all continuous $n$-linear operators between $E_1,\dots, E_m$ and $F$ is represented by $\mathcal{L}_m(E_1,\dots, E_m; F)$. If $E_1 = \cdots = E_m = E$ , we write simply $\mathcal{L}_m(^mE; F)$ and when $m=1$, we write $\mathcal{L}(E; F)$.

A mapping $P \colon E\rightarrow F$ is said to be {\it $m$-homogeneous polynomial} if there is an $A \in \mathcal{L}_m(^mE; F)$, such that $P(x) = A(x)^m$ for every $x \in E$, where $A(x)^m:= A\left(x,\overset{m}{\dots}, x \right)$. In this case, we write $P = \hat{A}$. By $\check{P}$ we denote the unique symmetric continuous $n$-linear operator associated
to P. For each positive integer $m$, we denote by $\mathcal{P}_m$ the class of all continuous $n$-homogeneous polynomial between Banach spaces.

\begin{definition}
	Let $m \in \mathbb{N}$. A Banach ideal of multilinear applications is a pair $\left(\mathcal{M}_m, \|\cdot \|_{\mathcal{M}_m} \right)$ where $\mathcal{M}_m$ is a  subclass of the class of all multilinear operators between Banach spaces and
	
	\begin{equation*}
		\|\cdot \|_{\mathcal{M}_m} : \mathcal{M}_m \longrightarrow \mathbb{R}
	\end{equation*}
	is a function such that, for all Banach spaces $E_1,\dots, E_m, F$, the component
	
	\begin{equation*}
		\mathcal{M}_m(E_1,\dots, E_m; F) := \mathcal{L}(E_1,\dots, E_m; F) \cap \mathcal{M}_m
	\end{equation*}
	is a subspace of $\mathcal{L}(E_1,\dots, E_m; F)$ on which $\|\cdot \|_{\mathcal{M}_m}$ is a complete norm and
	
	\begin{enumerate}
		\item The space of the multilinear operators of finite type is contained in $\mathcal{M}_m(E_1,\dots, E_m; F)$
		
		\item The application $I_m : \mathbb{K}^m \rightarrow \mathbb{K}$ given by $I_n(\lambda_1,\dots, \lambda_m) = \lambda_1 \cdots \lambda_m$ belongs  to $\mathcal{M}_m(\mathbb{K}^m; \mathbb{K})$ and
		
		\begin{equation*}
			\|I_m\|_{\mathcal{M}_m} = 1.
		\end{equation*}
		
		\item $($Multi-ideal property$)$ If $T \in \mathcal{M}_m(E_1,\dots, E_m; F), u_i \in \mathcal{L}(G_i; E_i), i=1,\dots, m$ and $t \in \mathcal{L}(F; H)$ then  $t\circ T \circ (u_1,\dots, u_m) \in \mathcal{M}_m(G_1,\dots, G_m; H)$ and
		
		\begin{equation*}
			\|t\circ T \circ (u_1,\dots, u_m) \|_{\mathcal{M}_m} \le \|t\| \left\|T \right\|_{\mathcal{M}_m} \|u_1\|\cdots \|u_m\|.
		\end{equation*}
	\end{enumerate}
\end{definition}

Analogously, we can define the homogeneous polynomial ideal $\mathcal{Q}$.

\begin{definition}
\begin{itemize}
\item  A class of vector-valued $m$-sequences $\gamma_s\left(\cdot; \mathbb{N}^m\right)$, or simply a $m$-sequence class $\gamma_s\left(\cdot; \mathbb{N}^m\right)$, is a rule that assigns to each $E \in BAN$ a Banach space $\gamma_s(E)$ of $E$-valued sequences;
	that is, $\gamma_s\left(\cdot; \mathbb{N}^m\right)$ is a vector subspace of $E^{\mathbb{N}\times\overset{m}{\cdots}\times \mathbb{N}}$ with coordinate wise operations, such that:
	$$c_{00}\left(E; \mathbb{N}^m\right)\subseteq \gamma_s\left(\cdot; \mathbb{N}^m\right)\overset{1}\hookrightarrow\ell_{\infty}\left(E; \mathbb{N}^m\right)\mbox{ and } \Vert e_{j_1,\dots, j_m}\Vert_{\gamma_s\left(\mathbb{K}; \mathbb{N}^m\right)}=1\mbox{ for every } j_1,\dots, j_m.$$
\item  A $m$-sequence class $\gamma_s\left(\cdot; \mathbb{N}^m\right)$ is {\it finitely determined} if for every sequence $(x_{j_1,\dots, j_m})_{j_1,\dots, j_m=1}^{\infty}\in E^{\mathbb{N}}$, $(x_{j_1,\dots, j_m})_{j_1,\dots, j_m=1}^{\infty}\in \gamma_s(E)$ if, and only if, $\sup_{k_1,\dots, k_m}\Vert (x_{j_1,\dots, j_m})_{j_1,\dots, j_m=1}^{k_1,\dots, k_m}\Vert_{\gamma_s\left(E; \mathbb{N}^m\right)}<+{\infty}$ and, in this case, $$\Vert (x_{j_1,\dots, j_m})_{{j_1,\dots, j_m}=1}^{\infty}\Vert_{\gamma_s(E)}=\sup_{k_1,\dots, k_m}\Vert (x_{j_1,\dots, j_m})_{j_1,\dots, j_m=1}^{k_1,\dots, k_m}\Vert_{\gamma_s(E)}.$$
\item  	A $m$-sequence class $\gamma_s\left(\cdot; \mathbb{N}^m\right)$ is said to be linearly stable if for every $u \in \mathcal{L}(E; F)$ 
	\begin{equation*}
	\left(u\left(x_{j_1,\dots, j_m} \right)\right)_{j_1,\dots, j_m=1}^{\infty} \in \gamma_s\left(F; \mathbb{N}^m\right)
	\end{equation*}
	 wherever $\left(x_{j_1,\dots, j_m} \right)_{j_1,\dots, j_m=1}^{\infty} \in \gamma_s\left(E; \mathbb{N}^m\right)$ and $\|\hat{u} : \gamma_s\left(E; \mathbb{N}^m\right) \rightarrow \gamma_s\left(F; \mathbb{N}^m\right)\| = \|u\|$.
\item  	Given sequence classes $\gamma_{s_1},\dots,\gamma_{s_m}$ and an $m$-sequence class $\gamma_s\left(\cdot; \mathbb{N}^m\right)$, we say that $\gamma_{s_1}(\mathbb{K})\cdots\gamma_{s_m}(\mathbb{K})\overset{multi, 1}\hookrightarrow\gamma_s\left(\mathbb{K}; \mathbb{N}^m\right)$ if $\left(\lambda_{j_1}^{(1)}\cdots\lambda_{j_m}^{(m)}\right)_{j_{j_1,\dots, m}=1}^{\infty}\in\gamma_s\left(\mathbb{K}; \mathbb{N}^m\right)$ and
	$$\left\Vert \left(\lambda_{j_1}^{(1)}\cdots\lambda_{j_m}^{(m)}\right)_{j_1,\dots, j_m=1}^{\infty}\right\Vert_{\gamma_s\left(\mathbb{K}; \mathbb{N}^m\right)}\le\prod_{m=1}^{n}\left\Vert\left(\lambda_{j}^{(m)}\right)_{j=1}^{\infty}\right\Vert_{\gamma_{s_m}(\mathbb{K})}$$
	whenever $\left(\lambda_{j}^{(i)}\right)_{j=1}^{\infty}\in\gamma_{s_i}(\mathbb{K})$, $i=1,\dots,m$.
\end{itemize}

These concepts recover the concepts introduced by Botelho and Campos in \cite{BC17}.

\begin{definition}\label{DDD.2.10.}
	Let $\gamma_s\left(\cdot; \mathbb{N}^m \right)$ be an $m$-sequence class, for every $m \in \mathbb{N}$ and let  $E$ be  a  Banach space. We say that the $m$-sequence class $\gamma_s\left(\cdot; \mathbb{N}^m \right)$ is sequentially compatible with $\gamma_s(\cdot) := \gamma_s\left(\cdot; \mathbb{N}^1 \right)$, if  for every $(x_{j_1,\dots, j_m})_{j_1,\dots, j_m=1}^{\infty} \in \gamma_s\left(E; \mathbb{N}^m \right)$, we have that  $(x_j)_{j=1}^{\infty} := (x_{j,\dots, j})_{j=1}^{\infty} \in \gamma_s(E)$ and
	\begin{equation*}
		\left\|(x_{j})_{j=1}^{\infty} \right\|_{\gamma_s(E)} \le \left\|(x_{j_1,\dots, j_m})_{j_1,\dots, j_m = 1}^{\infty} \right\|_{\gamma_s\left(E; \mathbb{N}^m \right)}.
	\end{equation*}
\end{definition}

A class that will play an important role in the course of this paper is the following: 
\begin{definition}
	Let $1 < p \le 2$. A continuous linear operator $u \colon E \rightarrow F$ is almost $p$-summing if $(u(x_j))_{j=1}^{\infty} \in Rad(F)$ whenever $(x_j)_{j=1}^{\infty} \in \ell_p^u(E)$. The class of all almost $p$-summing operators between Banach spaces is denoted by $\prod_{al, p}$.
\end{definition}
For more details regarding this class we refer to  \cite{DJT95}.

\end{definition}

\section{Multiple almost summing operators}

In this section, we will introduce the concept of multiple almost summing multilinear operators and present some of their properties.  

\begin{definition}
	Let $E$ be a Banach space. The linear space of all $m$-sequences $(x_{j_1,\dots, j_m})_{j_1,\dots, j_m = 1}^{\infty} \subset E^{\mathbb{N}^m}$ such that $\displaystyle\sum_{j_1,\dots, j_m = 1}^{n_1,\dots, n_m}r_{j_1}(t_1)\cdots r_{j_m}(t_m)x_{j_1,\dots, j_m}$ is convergent in $L_2([0, 1]^m; E)$ in almost all $(t_1,\dots, t_m) \in [0, 1]^m$ is denoted by $Rad\left(E; \mathbb{N}^m\right)$. 
\end{definition}

The $Rad\left(E; \mathbb{N}^m\right)$ is a Banach space when equipped with the norm 
\begin{equation*}
	\left\|(x_{j_1,\dots, j_m})_{j_1,\dots, j_m = 1}^{\infty}\right\|_{Rad\left(E; \mathbb{N}^m\right)} = \left(\int_{[0, 1]^m}\left\|\sum_{j_1,\dots, j_m = 1}^{\infty}r_{j_1}(t_1)\cdots r_{j_m}(t_m)x_{j_1,\dots, j_m} \right\|^2dt_1\cdots dt_m\right)^{\frac{1}{2}}.
\end{equation*}

\begin{remark}\label{RCL2}
	\begin{description}
		\item[(a)]  When $E = \mathbb{K}$, it is not difficult to see that $Rad\left(\mathbb{K}; \mathbb{N}^m\right) = \ell_2\left(\mathbb{K}; \mathbb{N}^m\right)$ and
		\begin{equation*}
			\left\|(x_{j_1,\dots, j_m})_{j_1,\dots, j_m = 1}^{\infty}\right\|_{Rad\left(\mathbb{K}; \mathbb{N}^m\right)} = \left\|(x_{j_1,\dots, j_m})_{j_1,\dots, j_m = 1}^{\infty}\right\|_{\ell_2\left(\mathbb{K}; \mathbb{N}^m\right)}
		\end{equation*}
	for every $(x_{j_1,\dots, j_m})_{j_1,\dots, j_m = 1}^{\infty} \in Rad\left(\mathbb{K}; \mathbb{N}^m\right)$.
		\item[(b)] If $(x_{j_1,\dots, j_m})_{j_1,\dots, j_m=1}^{\infty} \in Rad\left(E; \mathbb{N}^m\right)$ then
		\begin{align*}
			&\int_{[0, 1]^m}\left\|\sum_{j_1,\dots, j_m = 1}^{\infty}r_{j_1}(t_1)\cdots r_{j_m}(t_m)x_{j_1,\dots, j_m}\right\|^2dt_1\cdots dt_m\\
			&= \lim_{n_1,\dots, n_m \rightarrow \infty}\int_{[0, 1]^m}\left\|\sum_{j_1,\dots, j_m = 1}^{n_1,\dots, n_m}r_{j_1}(t_1)\cdots r_{j_m}(t_m)x_{j_1,\dots, j_m}\right\|^2dt_1\cdots dt_m.
		\end{align*}
	\end{description}
\end{remark}

Using the same notation and ideas as in  \cite[Contraction Principle $(12.2)$]{DJT95}, we can prove the following result.

\begin{proposition}\label{MPC}
	Let $1<p<\infty$ and let  $E$ be a Banach space. If $\chi_{j_1},\dots, \chi_{j_m}$, $j_1,\dots, j_m \in \mathbb{N}$, are independent symmetric real-valued random variables on a	probability space $(\Omega, \Sigma, P)$ and $\{x_{j_1,\dots, j_m}\}_{j_1,\dots, j_m=1}^{n_1,\dots, n_m} \subset E$, then, regardless of the choice of real
	numbers $a_{j_1,\dots, j_m}$
	\begin{align*}
		&\int_{\Omega^m}\left\|\sum_{j_1,\dots, j_m}^{n_1,\dots, n_m}a_{j_1,\dots, j_m}\chi_{j_1}(\omega_1)\cdots \chi_{j_m}(\omega_m)x_{j_1,\dots, j_m} \right\|^pdP(\omega_1)\cdots dP(\omega_m)\\
		&\le \left(\max_{j_1,\dots, j_m}|a_{j_1,\dots, j_m}|\right)\int_{\Omega^m}\left\|\sum_{j_1,\dots, j_m}^{n_1,\dots, n_m}\chi_{j_1}(\omega_1)\cdots \chi_{j_m}(\omega_m)x_{j_1,\dots, j_m} \right\|^pdP(\omega_1)\cdots dP(\omega_m).
	\end{align*}
In particular, if $A, B \subset \{j_1,\dots, j_m\}_{j_1,\dots, j_m=1}^{n_1,\dots, n_m}$ with $A \subset B$, then
\begin{align*}
	&\int_{\Omega^m}\left\|\sum_{j_1,\dots, j_m \in A}\chi_{j_1}(\omega_1)\cdots \chi_{j_m}(\omega_m)x_{j_1,\dots, j_m} \right\|^pdP(\omega_1)\cdots dP(\omega_m)\\
	&\le \int_{\Omega^m}\left\|\sum_{j_1,\dots, j_m \in B}\chi_{j_1}(\omega_1)\cdots \chi_{j_m}(\omega_m)x_{j_1,\dots, j_m} \right\|^pdP(\omega_1)\cdots dP(\omega_m).
\end{align*}
\end{proposition}

The next result, which can be prove as a consequence of the above proposition, will be important in verifying  whether  $Rad\left(\cdot; \mathbb{N}^m\right)$ is an $m$-sequence class.
\begin{lemma}\label{UESI}
	Let $\left(x_{j_1,\dots, j_m}\right)_{j_1,\dots, j_m=1}^{\infty} \in Rad\left(E; \mathbb{N}^m\right)$, then 
	\begin{equation*}
		\|x_{j_1,\dots, j_m}\|_E \le \left\|\left(x_{j_1,\dots, j_m}\right)_{j_1,\dots, j_m=1}^{\infty} \right\|_{Rad\left(E; \mathbb{N}^m\right)},
	\end{equation*}
for every $j_1,\dots, j_m \in \mathbb{N}$.
\end{lemma}


\begin{theorem}
	Let $m \in \mathbb{N}$ and $E$ be a Banach space. The class $E \mapsto Rad\left(E; \mathbb{N}^m\right)$ is a linearly stable $(n_1,\dots, n_m)$-sequence class.
\end{theorem}

\begin{proof}
	We will begin by showing that $Rad\left(\cdot; \mathbb{N}^m\right)$ is a $m$-sequence class. It is immediate that, for every Banach space $E$
	\begin{equation*}
		c_{00}\left(E; \mathbb{N}^m\right) \subset Rad\left(E; \mathbb{N}^m\right)
	\end{equation*}
	and by Remark \ref{RCL2} item (a)
\begin{equation*}
	\|e_{n_1,\dots, n_m}\|_{Rad\left(\mathbb{K}; \mathbb{N}^m\right)} = \|e_{j_1,\dots, j_m}\|_{\ell_2\left(\mathbb{K}; \mathbb{N}^m\right)} = 1.
\end{equation*}
Let $(x_{j_1,\dots, j_m})_{j_1,\dots, j_m = 1}^{\infty} \in Rad\left(E; \mathbb{N}^m\right)$, we already know that
\begin{equation*}
	\|x_{j_1,\dots, j_m}\| \le \left\|(x_{j_1,\dots, j_m})_{j_1,\dots, j_m = 1}^{\infty} \right\|_{Rad\left(E; \mathbb{N}^m\right)},
\end{equation*}
for every $j_1,\dots, j_m \in \mathbb{N}$. Then,
\begin{equation*}
	\sup_{j_1,\dots, j_m} \|x_{j_1,\dots, j_m}\| \le \left\|(x_{j_1,\dots, j_m})_{j_1,\dots, j_m = 1}^{\infty} \right\|_{Rad\left(E; \mathbb{N}^m\right)}
\end{equation*}
and therefore, $Rad\left(E; \mathbb{N}^m\right) \overset{1}{\hookrightarrow} \ell_{\infty}\left(E; \mathbb{N}^m\right)$. 

Now, we will prove that this class is linearly stable. Let $u \in \mathcal{L}(E; F)$ and $(x_{j_1,\dots, j_m})_{j_1,\dots, j_m = 1}^{\infty} \in Rad\left(E; \mathbb{N}^m\right)$. So,
\begin{align*}
	&\int_{[0,1]^m}\left\|\sum_{j_1,\dots, j_m = 1}^{\infty}r_{j_1}(t_1)\cdots r_{j_m}(t_m)u\left(x_{j_1,\dots, j_m}\right) \right\|^2dt_1\cdots dt_m\\
	&= \lim_{n_1,\dots, n_m \rightarrow \infty}\int_{[0,1]^m}\left\|\sum_{j_1,\dots, j_m = 1}^{n_1,\dots, n_m}r_{j_1}(t_1)\cdots r_{j_m}(t_m)u\left(x_{j_1,\dots, j_m}\right) \right\|^2dt_1\cdots dt_m\\
	&= \lim_{n_1,\dots, n_m \rightarrow \infty}\int_{[0,1]^m}\left\|u\left(\sum_{j_1,\dots, j_m = 1}^{n_1,\dots, n_m}r_{j_1}(t_1)\cdots r_{j_m}(t_m)x_{j_1,\dots, j_m}\right) \right\|^2dt_1\cdots dt_m\\
	&\le \|u\|^2\int_{[0,1]^m}\left\|\sum_{j_1,\dots, j_m = 1}^{n_1,\dots, n_m}r_{j_1}(t_1)\cdots r_{j_m}(t_m)x_{j_1,\dots, j_m} \right\|^2dt_1\cdots dt_m.
\end{align*}
Then, $(u\left(x_{j_1,\dots, j_m}\right))_{j_1,\dots, j_m = 1}^{\infty} \in Rad\left(F; \mathbb{N}^m\right)$. Consider the induced operator
\begin{equation*}
	\hat{u}\colon Rad\left(E; \mathbb{N}^m\right) \rightarrow Rad\left(F; \mathbb{N}^m\right)
\end{equation*}
given by $\hat{u}\left((x_{j_1,\dots, j_m})_{j_1,\dots, j_m=1}^{\infty}\right) = \left(u(x_{j_1,\dots, j_m})\right)_{j_1,\dots, j_m=1}^{\infty}$. Let  $(x_{j_1,\dots, j_m})_{j_1,\dots, j_m = 1}^{\infty} \in B_{Rad\left(E; \mathbb{N}^m\right)}$, by the above calculations
\begin{align*}
	\|\hat{u}\| &= \sup_{\|(x_{j_1,\dots, j_m})_{j_1,\dots, j_m = 1}^{\infty}\|_{Rad\left(E; \mathbb{N}^m\right)}=1}\|\hat{u}\left((x_{j_1,\dots, j_m})_{j_1,\dots, J_m=1}^{\infty}\right)\|_{Rad\left(F; \mathbb{N}^m\right)}\\
	&=\sup_{\|(x_{j_1,\dots, j_m})_{j_1,\dots, j_m = 1}^{\infty}\|_{Rad\left(E; \mathbb{N}^m\right)}=1}\left(\int_{[0,1]^m}\left\|\sum_{j_1,\dots, j_m = 1}^{\infty}r_{j_1}(t_1)\cdots r_{j_m}(t_m)u(x_{j_1,\dots, j_m}) \right\|^2dt_1\cdots dt_m\right)^{\frac{1}{2}}\\
	&\le \|u\|\sup_{\|(x_{j_1,\dots, j_m})_{j_1,\dots, j_m = 1}^{\infty}\|_{Rad\left(E; \mathbb{N}^m\right)}=1}\|(x_{j_1,\dots, j_m})_{j_1,\dots, J_m=1}^{\infty}\|_{Rad\left(E; \mathbb{N}^m\right)}\\
	&= \|u\|.
\end{align*}
The other inequality is as immediate consequence of Lemma \ref{UESI}.
\end{proof}

\begin{remark}
	As in the sequence case, this $(n_1,\dots, n_m)$-sequences class is not finitely determined.
\end{remark}

The concept of multiple $m$-linear almost summing was initially introduced in \cite{P06} and explored in other papers of  which we highlight \cite{P12, P14}. Let $1 < p_1,\dots, p_m \le 2$, a continuous $m$-linear operator $T \in \mathcal{L}(E_1,\dots, E_m; F)$ is said to be  multiple almost $(p_1,\dots, p_m)$-summing if 
\begin{equation*}
	\left(T\left(x_{j_1}^{(1)},\dots, x_{j_m}^{(m)}\right)\right)_{j_1,\dots, j_m=1}^{\infty} \in Rad\left(F; \mathbb{N}^n\right)
\end{equation*}
whenever $\left(x_j^{(i)}\right)_{j=1}^{\infty} \in \ell_{p_i}^u(E_i)$, $i=1,\dots, m$. The linear space of all  multiple almost $(p_1,\dots, p_m)$-summing operators from $E_1\times \cdots \times E_m$ to $F$ is denoted by $\mathcal{L}_{al, (p_1,\dots, p_m)}^{mult}(E_1,\dots, E_m; F)$. When $E_1 =\cdots= E_m=E$, we just write $\mathcal{L}_{al, (p_1,\dots, p_m)}^{mult}(^mE; F)$. When $p_1=\cdots=p_m=p${\bf, } we just write $\mathcal{L}_{al, p}^{mult}(E_1,\dots, E_m; F)$ and $\mathcal{L}_{al, p}^{mult}(^mE; F)$.

\begin{example}\label{EX1}
	Let $0 \neq \varphi_i \in E_i'$, $i=1,\dots, m$, $1 < p_1,\dots, p_{m+1} \le 2$ and $u \in \mathcal{\prod}_{al, p_{m+1}}(E_{m+1}; F)$. Define,
	\begin{equation*}
		\varphi_1 \otimes \cdots \otimes \varphi_m \otimes u \colon E_1\times \cdots \times E_{m+1} \rightarrow F
	\end{equation*}
given by, 
\begin{equation*}
	\varphi_1 \otimes \cdots \otimes \varphi_m \otimes u (x_1,\dots, x_m, x_{m+1}) = \varphi_1(x_1)\cdots \varphi_m(x_m) u(x_{m+1}).
\end{equation*}
The operator $\varphi_1 \otimes \cdots \otimes \varphi_m \otimes u$ is multiple almost $(p_1,\dots, p_{m+1})$-summing. Indeed, let $\left(x_{j}^{(i)}\right)_{j=1}^{\infty} \in \ell_{p_i}^u(E_i)$, $i=1,\dots, m+1$
\begin{align*}
	&\int_{[0, 1]^{m+1}}\left\|\sum_{j_1,\dots, j_{m+1}=1}^{n_1, \dots, n_{m+1}}r_{j_1}(t_1)\cdots r_{j_{m+1}}(t_{m+1})\varphi_1 \otimes \cdots \otimes \varphi_m \otimes u \left(x_{j_1}^{(1)},\dots, x_{j_{m+1}}^{(m+1)}\right) \right\|^2dt_1\cdots dt_{m+1}	\\
	=& \int_{[0, 1]^{m+1}}\left\|\sum_{j_1,\dots, j_{m+1}=1}^{n_1, \dots, n_{m+1}}r_{j_1}(t_1)\cdots r_{j_{m+1}}(t_{m+1})\varphi_1\left(x_{j_1}^{(1)}\right)\cdots \varphi_m\left(x_{j_m}^{(m)}\right)u\left(x_{j_{m+1}}^{(m+1)}\right) \right\|^2dt_1\cdots dt_{m+1}	\\
	=& \left(\int_{[0, 1]^{m}}\left|\sum_{j_1,\dots, j_{m+1}=1}^{n_1, \dots, n_m}r_{j_1}(t_1)\cdots r_{j_{m}}(t_{m})\varphi_1\left(x_{j_1}^{(1)}\right)\cdots \varphi_m\left(x_{j_m}^{(m)}\right) \right|^2dt_1\cdots dt_{m}\right)\\ &\times\left(\int_0^1\left\|\sum_{j_{m+1}}^{n_m}r_{m+1}(t_{m+1})u\left(x_{j_{m+1}}^{(m+1)}\right) \right\|^2dt_{m+1}\right)	\\
	&= \sum_{j_1,\dots, j_{m}=1}^{n_1,\dots, n_m}\left|\varphi_1\left(x_{j_1}^{(1)}\right)\cdots \varphi_m\left(x_{j_m}^{(m)}\right)\right|^2 \times \left(\int_0^1\left\|\sum_{j_{m+1}}^{n_m}r_{m+1}(t_{m+1})u\left(x_{j_{m+1}}^{(m+1)}\right) \right\|^2dt_{m+1}\right).
\end{align*}
Then, 
\begin{equation*}
	\left(\varphi_1 \otimes \cdots \otimes \varphi_m \otimes u \left(x_{j_1}^{(1)},\dots, x_{j_m}^{(m)}, x_{j_{m+1}}^{(m+1)}\right) \right)_{j_1,\dots, j_{m+1}=1}^{\infty} \in Rad\left(F; \mathbb{N}^{m+1}\right)
\end{equation*}
whenever $\left(x_{j}^{(i)}\right)_{j=1}^{\infty} \in \ell_{p_i}^u(E_i)$, $i=1,\dots, m+1$. Therefore, $\varphi_1 \otimes \cdots \otimes \varphi_m \otimes u$ is a multiple almost $(p_1,\dots, p_{m+1})$-summing multilinear operator.
\end{example}

In the next result, we will show a characterization of the concept of multiple almost summing multilinear operators.

\begin{proposition}\label{CONT}
	Let $1 < p_1,\dots, p_m \le 2$. A continuous $m$-linear operator $T \in \mathcal{L}(E_1,\dots, E_m; F)$ is multiple almost $(p_1,\dots, p_m)$-summing if and only if the induced operator
	\begin{equation*}
		\overset{\wedge}{T}\colon \ell_{p_1}^u(E_1)\times \cdots \times \ell_{p_m}^u(E_m) \rightarrow Rad\left(F; \mathbb{N}^n\right)
	\end{equation*}
given by 
\begin{equation*}
	\overset{\wedge}{T}\left(\left(x_j^{(1)}\right)_{j=1}^{\infty},\dots, \left(x_j^{(m)}\right)_{j=1}^{\infty}\right) = \left(T\left(x_{j_1}^{(1)},\dots, x_{j_m}^{(m)}\right)\right)_{j_1,\dots, j_m=1}^{\infty}
\end{equation*}
is an well-defined and continuous $m$-linear operator.
\end{proposition}

\begin{proof}
	It is immediate from the definition of multiple almost $(p_1,\dots, p_m)$-summing operators that $\hat{T}$ is an  well-defined $m$-linear operator. Now, we will prove the continuity. 
	
	Let
	\begin{equation*}
		\left(\left(x_j^{k, (1)}\right)_{j=1}^{\infty},\dots, \left(x_j^{k, (m)}\right)_{j=1}^{\infty}, \left(T\left(x_{j_1}^{k, (1)},\dots, x_{j_m}^{k, (m)}\right)\right)_{j_1,\dots, j_m = 1}^{\infty}\right)_{k=1}^{\infty}
	\end{equation*}
be a sequence on the graph of $T$  such that
\begin{equation*}
	\left(x_j^{k, (i)}\right)_{j=1}^{\infty} \underset{k}{\longmapsto} \left(x_j^{(i)}\right)_{j=1}^{\infty} \text{, } i=1,\dots, m
\end{equation*}
and
\begin{equation*}
	T\left(x_{j_1}^{k, (1)},\dots, x_{j_m}^{k, (m)}\right)_{j_1,\dots, j_m=1}^{\infty} \underset{k}{\longmapsto} \left(z_{j_1,\dots, j_m}\right)_{j_1,\dots, j_m=1}^{\infty}.
\end{equation*}
So, for every $i = \{1,\dots, m\}$
\begin{equation*}
	x_j^{k, (i)} \underset{k}{\longmapsto} x_j^{(i)} \text{ and } T\left(x_{j_1}^{k, (1)},\dots, x_{j_m}^{k, (m)}\right) \underset{k}{\longmapsto} z_{j_1,\dots, j_m}\text{, } j_1,\dots, j_m \in \mathbb{N}.
\end{equation*}
By the continuity of $T$ and the uniqueness of the limit, we have that
\begin{equation*}
	T\left(x_{j_1}^{(1)},\dots, x_{j_m}^{(m)}\right) = z_{j_1,\dots, j_m}
\end{equation*}
for every $j_1,\dots, j_m \in \mathbb{N}$. Then,
\begin{equation*}
	\left(T\left(x_{j_1}^{(1)},\dots, x_{j_m}^{(m)}\right)\right)_{j_1,\dots, j_m=1}^{\infty} = \left(z_{j_1,\dots, j_m}\right)_{j_1,\dots, j_m=1}^{\infty}.
\end{equation*}
Thus,  $\hat{T}$ has a closed graph and therefore it is continuous.
\end{proof}

We define a norm in $\mathcal{L}_{al, (p_1,\dots, p_m)}^{mult}(E_1,\dots, E_m; F)$ by
\begin{equation*}
	\|T\|_{\mathcal{L}_{al, (p_1,\dots, p_m)}^{mult}} := \left\|\overset{\wedge}{T}\right\|.
\end{equation*}

\begin{remark}\label{Norm}
	For every $T \in \mathcal{L}_{al, (p_1,\dots, p_m)}^{mult}(E_1,\dots, E_m; F)$ we have that
	\begin{equation*}
		\|T\| \le \|T\|_{\mathcal{L}_{al, (p_1,\dots, p_m)}^{mult}}.
	\end{equation*}
\end{remark}

\begin{theorem}\label{COMPL}
	Let $0<p_1,\dots, p_m<\infty$ and $E_1,\dots, E_m, F$ be Banach spaces. Then, $\mathcal{L}_{al, (p_1,\dots, p_m)}^{mult}(E_1,\dots, E_m; F)$ is a Banach space.
\end{theorem}

\begin{proof}
	Let $(T_k)_{k=1}^{\infty}$ be a Cauchy sequence in  $\mathcal{L}_{al, (p_1,\dots, p_m)}^{mult}(E_1,\dots, E_m; F)$, so
	\begin{equation*}
		\left\|T_k - T_l \right\|_{\mathcal{L}_{al, (p_1,\dots, p_m)}^{mult}} \underset{k, l}{\longrightarrow} 0.
	\end{equation*}
Thus, 
\begin{equation*}
	\left\|\overset{\wedge}{T}_k-\overset{\wedge}{T}_l \right\| = \left\|(T_k-T_l)^{\wedge} \right\| \underset{k, l}{\rightarrow} 0 {\bf \text{.}}
\end{equation*}
So, $(\overset{\wedge}{T}_k)_{k=1}^{\infty} \subset \mathcal{L}\left(\ell_{p_1}^u(E_1),\dots, \ell_{p_m}^u(E_m); Rad\left(F; \mathbb{N}^m\right)\right)$ is a Cauchy sequence{\bf .} Then there is $B \in \mathcal{L}\left(\ell_{p_1}^u(E_1),\dots, \ell_{p_m}^u(E_m); Rad\left(F; \mathbb{N}^m\right)\right)$ such that
\begin{equation*}
	\overset{\wedge}{T}_k \underset{k}{\longrightarrow} B.
\end{equation*}
Thus,
\begin{align*}
	&\left\|B\left(\left(x_j^{(1)}\right)_{j=1}^{\infty},\dots, \left(x_j^{(m)}\right)_{j=1}^{\infty}\right) -  \overset{\wedge}{T}_k\left(\left(x_j^{(1)}\right)_{j=1}^{\infty},\dots, \left(x_j^{(m)}\right)_{j=1}^{\infty}\right)\right\|_{Rad\left(F; \mathbb{N}^m\right)}\\
	&= \left\|B\left(\left(x_j^{(1)}\right)_{j=1}^{\infty},\dots, \left(x_j^{(m)}\right)_{j=1}^{\infty}\right) -  \left(T_k\left(x_{j_1}^{(1)},\dots, x_{j_m}^{(m)}\right)\right)_{j_1,\dots, j_m=1}^{\infty}\right\|_{Rad\left(F; \mathbb{N}^m\right)} \underset{k}{\longrightarrow} 0,
\end{align*}
for every $\left(x_j^{(i)}\right)_{j=1}^{\infty} \in \ell_{p_i}^u(E_i)$, $i=1,\dots, m$.
We call $B\left(\left(x_j^{(1)}\right)_{j=1}^{\infty},\dots, \left(x_j^{(m)}\right)_{j=1}^{\infty}\right)$ by $\left(z_{j_1,\dots, j_m}^{x_{j_1}^{(1)},\dots, x_{j_m}^{(m)}}\right)_{j_1,\dots, j_m=1}^{\infty}$. It follows from Lemma \ref{UESI} that
\begin{align*}
	\left\|z_{j_1,\dots, j_m}^{x_{j_1}^{(1)},\dots, x_{j_m}^{(m)}} - T_k\left(x_{j_1}^{(1)},\dots, x_{j_m}^{(m)}\right)\right\| \underset{k}{\longrightarrow} 0,
\end{align*}
so 
\begin{equation*}
	T_k\left(x_{j_1}^{(1)},\dots, x_{j_m}^{(m)}\right) \underset{k}{\longrightarrow} z_{j_1,\dots, j_m}^{x_{j_1}^{(1)},\dots, x_{j_m}^{(m)}}\text{, } \forall j_1,\dots, j_m \in \mathbb{N}.
\end{equation*}
On the other hand, since  $(T_k)_{k=1}^{\infty}$ is a Cauchy sequence in  $\mathcal{L}_{al, (p_1,\dots, p_m)}^{mult}(E_1,\dots, E_m; F)$, by Remark \ref{Norm} $(T_k)_{k=1}^{\infty}$ is a Cauchy sequence in  $\mathcal{L}(E_1,\dots, E_m; F)$ and there is $T \in \mathcal{L}(E_1,\dots, E_m; F)$ such that
\begin{equation*}
	T_k\left(x_{j_1}^{(1)},\dots, x_{j_m}^{(m)}\right) \underset{k}{\longrightarrow} T\left(x_{j_1}^{(1)},\dots, x_{j_m}^{(m)}\right).
\end{equation*}
Then, 
\begin{equation*}
	T\left(x_{j_1}^{(1)},\dots, x_{j_m}^{(m)}\right) = z_{j_1,\dots, j_m}^{x_{j_1}^{(1)},\dots, x_{j_m}^{(m)}},
\end{equation*}
thus
\begin{equation*}
	\left(T\left(x_{j_1}^{(1)},\dots, x_{j_m}^{(m)}\right) \right)_{j_1,\dots, j_m=1}^{\infty} = \left( z_{j_1,\dots, j_m}^{x_{j_1}^{(1)},\dots, x_{j_m}^{(m)}}\right)_{j_1,\dots, j_m=1}^{\infty}.
\end{equation*}
So, if $\left(x_j^{(i)}\right)_{j=1}^{\infty} \in \ell_{p_i}^u\left(E_i\right)$ for $i=1,\dots, m$, then 
\begin{equation*}
	\left(T\left(x_{j_1}^{(1)},\dots, x_{j_m}^{(m)}\right)\right)_{j_1,\dots, j_m=1}^{\infty} = \left( z_{j_1,\dots, j_m}^{x_{j_1}^{(1)},\dots, x_{j_m}^{(m)}}\right)_{j_1,\dots, j_m=1}^{\infty} \in Rad\left(F; \mathbb{N}^m\right) {\bf \text{.}} 
\end{equation*}
Thus, $T \in \mathcal{L}_{al, (p_1,\dots, p_m)}^{mult}(E_1,\dots, E_m; F)$.
\end{proof}

\begin{remark}\label{INC1}
	Let $m \in \mathbb{N}$ and $1 < p_1,\dots, p_m \le 2$. Then, $\ell_{p_1}^u(\mathbb{K})\cdots \ell_{p_m}^u(\mathbb{K}) \overset{multi, 1}{\hookrightarrow} Rad\left(\mathbb{K}; \mathbb{N}^m\right)$. Indeed, by Remark \ref{RCL2} item (a), for every $\left(\lambda_j^{(i)}\right)_{j=1}^{\infty} \in \ell_{p_i}^u(E_i)$, $i=1,\dots, m$
	\begin{align*}
		&\int_{[0, 1]^m}\left\|\sum_{j_1,\dots, j_m = 1}^{\infty}r_{j_1}(t_1)\cdots r_{j_m}(t_m)\lambda_{j_1}^{(1)}\cdots \lambda_{j_m}^{(m)}\right\|^2dt_1\cdots dt_m = \sum_{j_1,\dots, j_m=1}^{\infty}\left|\lambda_{j_1}^{(1)}\cdots \lambda_{j_m}^{(m)}\right|^2 =\\
		&= \prod_{i=1}^{m}\left(\sum_{j=1}^{\infty}\left|\lambda_{j}^{(i)} \right|^2\right)
		\le \prod_{i=1}^{m}\left( \sum_{j = 1}^{\infty}\left|\lambda_{j
		}^{(i)} \right|^{p_i}\right).
	\end{align*}
	So, 
	\begin{equation*}
		\left(\lambda_{j_1}^{(1)}\cdots \lambda_{j_m}^{(m)}\right)_{j_1,\dots, j_m=1}^{\infty} \in Rad\left(\mathbb{K}; \mathbb{N}^m\right)
	\end{equation*}
and
\begin{equation*}
	\left\|\left(\lambda_{j_1}^{(1)}\cdots \lambda_{j_m}^{(m)}\right)_{j_1,\dots, j_m=1}^{\infty} \right\|_{Rad\left(\mathbb{K}; \mathbb{N}^m\right)} \le \prod_{i=1}^{m}\left\|\left(\lambda_j^{(i)}\right) \right\|_{p_i} = \prod_{i=1}^{m}\left\|\left(\lambda_j^{(i)}\right) \right\|_{w, p_i}.
\end{equation*}
\end{remark}


A consequence of Remark \ref{INC1} is the following result.

\begin{proposition}\label{TFID1}
	\begin{description}
		\item[(a)] 	Let $1 < p_1,\dots, p_m \le 2$ and $E_1,\dots, E_m, F$ be Banach spaces. Then, the $m$-linear operators of finite type are  contained in $\mathcal{L}_{al, (p_1,\dots, p_m)}^{mult}(E_1,\dots, E_m; F)$.
		\item[(b)] 	Let $1 < p_1,\dots, p_m \le 2$ and $T\colon \mathbb{K}^m \rightarrow \mathbb{K}$ be given by $T(x_1,\dots, x_m) = x_1\cdots x_m$, then $\left\|T \right\|_{\mathcal{L}_{al, (p_1,\dots, p_m)}^{mult}} = 1$. 
	\end{description}

\end{proposition}


Using the results above, we can prove the main result of this section.

\begin{theorem}
	Let $1 < p_1,\dots, p_m \le 2$. Then the class $\left(\mathcal{L}_{al, (p_1,\dots, p_m)}^{mult}, \|\cdot \|_{\mathcal{L}_{al, (p_1,\dots, p_m)}^{mult}}\right)$ is a Banach multi-ideal between Banach spaces.
\end{theorem}

\begin{proof}
	Since we already  have the  Proposition \ref{TFID1} and Theorem \ref{COMPL}, we need to show only the property of multi-ideal. Let $t \in \mathcal{L}(F, H)$, $T \in \mathcal{L}_{al, (p_1,\dots, p_m)}^{mult}(E_1,\dots, E_m; F)$ and $u_i \in \mathcal{L}(G_i; E_i)$, $i=1,\dots, m$. Then, 
	\begin{align*}
		&\int_{[0, 1]^m}\left\|\sum_{j_1,\dots, j_m = 1}^{\infty}r_{j_1}(t_1)\cdots r_{j_m}(t_m) t \circ T \circ (u_1,\dots, u_m)\left(x_{j_1}^{(1)},\dots, x_{j_m}^{(m)}\right) \right\|^2dt_1\cdots dt_m\\
		&= \lim_{n_1,\dots, n_m\rightarrow \infty} \int_{[0, 1]^m}\left\|\sum_{j_1,\dots, j_m = 1}^{n_1,\dots, n_m}r_{j_1}(t_1)\cdots r_{j_m}(t_m) t \circ T \circ (u_1,\dots, u_m)\left(x_{j_1}^{(1)},\dots, x_{j_m}^{(m)}\right) \right\|^2dt_1\cdots dt_m\\
		&= \lim_{n_1,\dots, n_m\rightarrow \infty} \int_{[0, 1]^m}\left\|t\left(\sum_{j_1,\dots, j_m = 1}^{n_1,\dots, n_m}r_{j_1}(t_1)\cdots r_{j_m}(t_m) T \circ (u_1,\dots, u_m)\left(x_{j_1}^{(1)},\dots, x_{j_m}^{(m)}\right)\right) \right\|^2dt_1\cdots dt_m\\
		&\le \|t\|^2 \lim_{n_1,\dots, n_m\rightarrow \infty} \int_{[0, 1]^m}\left\|\sum_{j_1,\dots, j_m = 1}^{n_1,\dots, n_m}r_{j_1}(t_1)\cdots r_{j_m}(t_m) T\left(u_1\left(x_{j_1}^{(1)}\right),\dots, u_m\left(x_{j_m}^{(m)}\right)\right) \right\|^2dt_1\cdots dt_m.
	\end{align*}
	As $Rad\left(\cdot; \mathbb{N}^m\right)$ and $\ell_{p_i}^{(i)}(\cdot)$ for $i=1,\dots, m$ are linearly stable, it follows that $t \circ T \circ (u_1,\dots, u_m) \in \mathcal{L}_{al, (p_1,\dots, p_m)}^{mult}(G_1,\dots, G_m; H)$ and,
	\begin{align*}
		&\left\|\left(t \circ T \circ (u_1,\dots, u_m)\right)^{\wedge}\left(\left(x_j^{(1)}\right)_{j=1}^{\infty},\dots, \left(x_j^{(m)}\right)_{j=1}^{\infty}\right) \right\|_{Rad\left(H; \mathbb{N}^m\right)}\\
		&\le \|t\| \left\|\left(T\left(u_1\left(x_{j_1}^{(1)}\right),\dots, u_m\left(x_{j_m}^{(m)}\right)\right)\right)_{j_1,\dots, j_m=1}^{\infty} \right\|_{Rad\left(F; \mathbb{N}^m\right)}\\
		&= \|t\| \left\|\overset{\wedge}{T}\left(\left(u_1\left(x_j^{(1)}\right)\right)_{j=1}^{\infty},\dots, \left(u_1\left(x_j^{(m)}\right)\right)_{j=1}^{\infty} \right) \right\|\\
		&\le \|t\| \left\|\overset{\wedge}{T} \right\| \|u_1\|\cdots \|u_m\| \prod_{i=1}^{m}\left\|\left(x_j^{(i)}\right)_{j=1}^{\infty} \right\|_{w, p_i}\\
		&= \|t\| \left\|T \right\|_{\mathcal{L}_{al, (p_1,\dots, p_m)}^{mult}}\|u_1\|\cdots \|u_m\|\prod_{i=1}^{m}\left\|\left(x_j^{(i)}\right)_{j=1}^{\infty} \right\|_{w, p_i}.
	\end{align*}
	Therefore,
	\begin{align*}
	\left\|t \circ T \circ (u_1,\dots, u_m) \right\|_{\mathcal{L}_{al, (p_1,\dots, p_m)}^{mult}} = \left\|\left(t \circ T \circ (u_1,\dots, u_m)\right)^{\wedge} \right\| \le  \|t\| \left\|T \right\|_{\mathcal{L}_{al, (p_1,\dots, p_m)}^{mult}} \|u_1\|\cdots \|u_m\|.
	\end{align*}
\end{proof}

\section{Coherence and compatibility}

The concept of Coherence and compatibility that we use in this section was intrduced by Pelegrino and Ribeiro in \cite{PR14}. In this concept a pair of multi-ideals and homogeneous polynomial  ideals is used, so it is necessary to define a class of homogeneous  polynomials associated  to $\mathcal{L}_{al, (p_1,\dots, p_m)}^{mult}$. Let's remember these concepts.

\begin{definition}[Compatible pair of ideals]
	Let $\mathcal{U}$ be a normed operator ideal and $N \in \left(\mathbb{N} - \{1\}  \right)\cup \{\infty \}$. A sequence $\left(\mathcal{U}_n, \mathcal{M}_n  \right)_{n=1}^N$, with $\mathcal{U}_1 = \mathcal{M}_1 = \mathcal{U}$, is compatible with $\mathcal{U}$  if there exist positive constants $\alpha_1, \alpha_2, \alpha_3$ such that for all Banach spaces $E$ and $F$, the following conditions hold for all $n \in \{2,\cdots, N\}:$
	
	\begin{description}
		\item $(CP1)$  If $k \in \{1,\dots, n\}$, $T \in \mathcal{M}_n(E_1,\dots, E_n;F)$ and $a_j \in E_j$ for all $j \in \{1,\dots, n\}\setminus\{k\}$, then $ T_{a_1,\dots, a_{k-1},a_{k+1},\dots, a_n} \in \mathcal{U}(E_k; F)$
		and
		\begin{equation*}
			\left\Vert T_{a_1,\dots, a_{k-1},a_{k+1},\dots, a_n} \right\Vert \le \alpha_1 \left\Vert T\right\Vert_{\mathcal{M}_n}\|a_1\|\cdots \|a_{k-1}\| \ \|a_{k+1}\|\cdots \|a_n\|.
		\end{equation*}
		
		\item $(CP2)$ If $P \in \mathcal{U}_n(^nE; F)$ and $a \in E$, then $P_{a^{n-1}} \in \mathcal{U}(E; F)$ and
		\begin{equation*}
			\left\Vert P_{a^{n-1}}\right\Vert_{\mathcal{U}} \le \alpha_2 \max{\left\{\left\Vert\overset{\vee}{P}\right\Vert_{\mathcal{M}_n}, \left\Vert P \right\Vert_{\mathcal{U}_n} \right\}}\Vert a\Vert^{n-1}.
		\end{equation*}
		
		\item (CP3) If $u \in \mathcal{U}(E_n; F)$, $\gamma_j \in E'_j$ for all $j = 1,\dots, n-1$, then $\gamma_1 \cdots \gamma_{n-1}u \in \mathcal{M}_n(E_1,\dots, E_n; F)$
		and 
		\begin{equation*}
			\left\Vert\gamma_1 \cdots\gamma_{n-1}u  \right\Vert_{\mathcal{M}_n} \le \alpha_3 \Vert\gamma_1\Vert\cdots\Vert\gamma_{n-1}\Vert\left\Vert u  \right\Vert_{\mathcal{U}}.
		\end{equation*}
		
		\item $(CP4)$ If $u \in \mathcal{U}(E; F)$ and $\gamma \in E'$, then $\gamma^{(n-1)}u \in \mathcal{U}_n(^{n}E; F)$.
		
		\item $(CP5)$ $P$ belongs to $\mathcal{U}_n(^nE; F)$ if and only if  $\overset{\vee}{P}$ belongs to $\mathcal{M}_n(^nE; F)$.
	\end{description}
	
\end{definition}

\begin{definition}[Coherent pair of ideals]\label{D2.2.}
	Let $\mathcal{U}$ be a normed operator ideal and let $N \in \mathbb{N} \cup \{\infty \}$. A sequence $\left(\mathcal{U}_k, \mathcal{M}_k  \right)_{k=1}^{N}$, with $\mathcal{U}_1 = \mathcal{M}_1 = \mathcal{U}$, is coherent if there exist positive constants $\beta_1, \beta_2, \beta_3$ such that for all Banach spaces $E$ and $F$ the following conditions hold for $k = 1,\dots,N-1:$
	
	\begin{description}
		\item $(CH1)$ If $T \in \mathcal{M}_{k + 1}\left(E_1,\dots, E_{k+1}; F \right)$ and $a_j \in E_j$ for $j=1,\dots, k+1$, then
		\begin{equation*}
			T_{a_j} \in \mathcal{M}_k\left(E_1,\dots, E_{j-1}, E_{j+1},\dots, E_{k+1}; F  \right)
		\end{equation*}
		and
		\begin{equation*}
			\left\|T_{a_j} \right\|_{\mathcal{M}_k} \le \beta_1 \left\| T \right\|_{\mathcal{M}_{k + 1}}\|a_j\|.
		\end{equation*}
		
		\item $(CH2)$ If $P \in \mathcal{U}_{k+1}\left(^{k+1}E; F \right)$, $a \in E$, then $P_a$ belongs to $\mathcal{U}_k\left(^kE; F \right)$ and
		\begin{equation*}
			\left\| P_a \right\|_{\mathcal{U}_k} \le \beta_2 \max{\left\{\left\| \overset{\vee}{P} \right\|_{\mathcal{M}_{k+1}}, \left\| P \right\|_{\mathcal{U}_{k+1}} \right\}}\Vert a \Vert.
		\end{equation*}
		
		\item $(CH3)$ If $T \in \mathcal{M}_k(E_1,\dots,E_k; F)$, $\gamma \in E'_{k+1}$, then
		\begin{equation*}
			\gamma T \in \mathcal{M}_{k + 1}(E_1,\dots,E_{k + 1}; F)
		\end{equation*}
		and 
		\begin{equation*}
			\left\|\gamma T\right\|_{\mathcal{M}_{k + 1}} \le \beta_3\Vert\gamma\Vert \left\|T \right\|_{ \mathcal{M}_{k}}.
		\end{equation*}
		
		\item $(CH4)$ If $P \in \mathcal{U}_{k}\left(^kE; F \right)$ and $\gamma \in E'$, then $\gamma P \in \mathcal{U}_{k+1}\left(^{k + 1}E; F \right).$
		
		\item $(CH5)$ For all $k=1,\dots,N$, $P$ belongs to $\mathcal{U}_k(^kE; F)$ if and only if $\overset{\vee}{P}$ belongs to $\mathcal{M}_k(^kE; F)$.
	\end{description}
	
\end{definition}

In this section, the class of all multiple almost $p$-summing $m$-linear operators is denoted by $\mathcal{L}_{al, p}^{multi, m}$. The reason for this is to highlight the linearity of the operators. The next result is folklore and introduces  the class of homogeneous polynomials that will be considered in this section. 

\begin{proposition}\label{IHPG}
	Let $\left(\mathcal{M}, \|\cdot\|_{\mathcal{M}}\right)$ be a Banach multi-ideal. Then,	
	\begin{equation*}
		\mathcal{P}_{\mathcal{M}}:= \{P \in \mathcal{P}\text{; } \check{P} \in \mathcal{M}\}\text{, } \|P\|_{\mathcal{P}_{\mathcal{M}}} := \|\check{P}\|_{\mathcal{M}}
	\end{equation*}
is a Banach ideal of homogeneous polynomials.
\end{proposition}

The class $(\mathcal{P}_{\mathcal{M}}, \|\cdot\|_{\mathcal{P}_{\mathcal{M}}})$ is called the Banach ideal of homogeneous polynomials generated by $\left(\mathcal{M}, \|\cdot\|_{\mathcal{M}}\right)$. This class has
been studied extensively in several works, of which we highlight \cite{BJ, B05, FG03}.

\begin{definition}
	Let $1 < p \le 2$. The class of all multiple almost $p$-summing $m$-homogeneous polynomials and its norm are defined by
	\begin{equation*}
	\mathcal{P}_{al, p}^{multi, m} :=	\mathcal{P}_{\mathcal{L}_{al, p}^{multi, m}} = \left\{P \in \mathcal{P}\text{; } \check{P} \in \mathcal{L}_{al, p}^{multi, m} \right\} { and }\ \|P\|_{\mathcal{P}_{al, p}^{multi, m}} := \|\check{P}\|_{\mathcal{L}_{al, p}^{multi, m}}.
	\end{equation*}
\end{definition}

As $\left(\mathcal{L}_{al, p}^{multi, m}, \|\cdot\|_{\mathcal{L}_{al, p}^{multi, m}}\right)$ is a Banach Multi-ideal, {\bf it} follows from Proposition \ref{IHPG} that
\begin{equation*}
	\left(\mathcal{P}_{al, p}^{multi, m}, \|\cdot\|_{\mathcal{P}_{al, p}^{multi, m}} \right)
\end{equation*}
is a Banach ideal of homogeneous polynomials.

An immediate consequence of the Polarization Formula is the following:

\begin{lemma}\label{PMD}
	For every $P \in \mathcal{P}(E; F)$ and $a \in E${\bf, }  $(P_a)^{\vee} = \check{P}_{a}$.
\end{lemma}

Now, we can prove the main result in this section.

\begin{theorem}\label{CCMAS}
	Let $1<p\le 2$. Then, the sequence of pairs
	\begin{equation*}
		\left(\left(\mathcal{L}_{al, p}^{multi, m}, \|\cdot\|_{\mathcal{L}_{al, p}^{multi, m}}\right), \left(\mathcal{P}_{al, p}^{multi, m}, \|\cdot\|_{\mathcal{P}_{al, p}^{multi, m}} \right) \right)
	\end{equation*}
is coherent and compatible with $\prod_{al, p}$.
\end{theorem}

\begin{proof}
	To prove the coherence, it is necessary to check  properties  $(CH1),\dots, (CH5)$. We begin by showing property (CH1). We will do only the case $i=1$. The other cases are analogous. Let $T \in \mathcal{L}_{al, p}^{multi, m}(E_1,\dots, E_m; F)$ and $a_1 \in E_1$.
	
	Consider $(x_j^{(1)})_{j=1}^{\infty} \in \ell_p^u(E_1)$ given by $x_1^{(1)} = a_1$ and $x_j^{(1)} = 0$ for every $j\in \mathbb{N}\setminus\{1\}$. Thus,
	\begin{equation*}
		\left(T\left(x_{j_1}^{(1)},\dots, x_{j_m}^{(m)}\right)\right)_{j_1,\dots, j_m=1}^{\infty} \in Rad\left(F; \mathbb{N}^m\right),
	\end{equation*}
for every $\left(x_j^{(i)}\right)_{j=1}^{\infty} \in \ell_{p_i}^u(E_i)$, $i=2,\dots, m$. So, 
\begin{align*}
	&\int_{[0, 1]^{m-1}}\left\|\sum_{j_2,\dots, j_m=1}^{\infty}r_{j_2}(t_2)\cdots r_{j_m}(t_m)T_{a_1}\left(x_{j_2}^{(2)},\dots, x_{j_m}^{(m)}\right) \right\|^2dt_2\cdots dt_m\\
	&= \int_{[0, 1]^m}\left\|\sum_{j_1,\dots, j_m=1}^{\infty}r_{j_1}(t_1)\cdots r_{j_m}(t_m)T\left(x_{j_1}^{(1)},\dots, x_{j_m}^{(m)}\right) \right\|^2dt_1\cdots dt_m < \infty.
\end{align*}
Then, 
\begin{equation*}
\left(T_{a_1}\left(x_{j_2}^{(2)},\dots, x_{j_m}^{(m)}\right)\right)_{j_2,\dots, j_m=1}^{\infty} \in Rad\left(F; \mathbb{N}^{m-1}\right),
\end{equation*}
for every $\left(x_j^{(i)}\right)_{j=1}^{\infty} \in \ell_p^u(E_i)$, $i=2,\dots, m$. Therefore, $T_{a_1} \in \mathcal{L}_{al, p}^{mult, m-1}(E_2,\dots, E_m; F)$ and 
\begin{align*}
	\left\|T_{a_1} \right\|_{\mathcal{L}_{al, m}^{multi, m-1}} = \left\|(T_{a_1})^{\wedge} \right\| &= \sup_{\left\|\left(x_j^{(i)}\right)_{j=1}^{\infty} \right\|_{w, p} \le 1}\left\|(T_{a_1})^{\wedge}\left(\left(x_{j}^{(2)}\right)_{j}^{\infty},\dots, \left(x_{j}^{(m)}\right)_{j=1}^{\infty}\right) \right\|_{Rad\left(F; \mathbb{N}^{m-1}\right)}\\
	&= \sup_{\left\|\left(x_j^{(i)}\right)_{j=1}^{\infty} \right\|_{w, p} \le 1}\left\|\left(T_{a_1}\left(x_{j_2}^{(2)},\dots, x_{j_m}^{(m)}\right)\right)_{j_2,\dots, j_m=1}^{\infty} \right\|_{Rad\left(F; \mathbb{N}^{m-1}\right)}\\
	&\le \sup_{\left\|\left(x_j^{(i)}\right)_{j=1}^{\infty} \right\|_{w, p} \le 1}\left\|\left(T\left(x_{j_1}^{(1)},\dots, x_{j_m}^{(m)}\right)\right)_{j_2,\dots, j_m=1}^{\infty} \right\|_{Rad\left(F; \mathbb{N}^{m}\right)}\\
	&= \sup_{\left\|\left(x_j^{(i)}\right)_{j=1}^{\infty} \right\|_{w, p} \le 1}\left\|\overset{\wedge}{T}\left(\left(x_{j}^{(1)}\right)_{j}^{\infty},\dots, \left(x_{j}^{(m)}\right)_{j=1}^{\infty}\right) \right\|_{Rad\left(F; \mathbb{N}^{m}\right)}\\
	&\le \sup_{\left\|\left(x_j^{(i)}\right)_{j=1}^{\infty} \right\|_{w, p} \le 1} \left\|\overset{\wedge}{T} \right\|\left\|\left(x_j^{(1)}\right) \right\|_{w, p}\prod_{i=2}^{\infty}\left\|\left(x_j^{(i)}\right) \right\|_{w, p}\\
	&= \left\|T \right\|_{\mathcal{L}_{al, m}^{multi, m}}\|a_1\|.
\end{align*}

Now,  we check (CH2). Let $P \in \mathcal{P}_{al, p}^{multi, m}(E; F)$ and $a \in E$. By definition, $\check{P} \in \mathcal{L}_{al, p}^{multi, m}\left(E^m; F\right)$, thus by (CH1)
\begin{equation*}
	\check{P}_a \in  \mathcal{P}_{al, p}^{multi, m-1}(E; F).
\end{equation*}
By Lemma \ref{PMD},
\begin{equation*}
	(P_a)^{\vee} = \check{P}_a \in  \mathcal{P}_{al, p}^{multi, m-1}(E; F).
\end{equation*}
So, $P_a \in \mathcal{P}_{al, p}^{multi, m-1}(E; F)$ and
\begin{equation*}
	\|P_a\|_{\mathcal{P}_{al, p}^{multi, m-1}} = \|(P_a)^{\vee}\|_{\mathcal{L}_{al, p}^{multi, m-1}} = \|\check{P}_a\|_{\mathcal{L}_{al, p}^{multi, m-1}} \le \|\check{P}\|_{\mathcal{L}_{al, p}^{multi, m}}\|a\| = \|P\|_{\mathcal{P}_{al, p}^{multi, m}}\|a\|.
\end{equation*}

Now,  we check (CH3). Let $T \in \mathcal{L}_{al, p}^{multi, m}(E_1,\dots, E_m; F)$ and $\varphi \in E_{m+1}'$, then
for every $\left(x_j^{(i)}\right)_{j=1}^{\infty} \in \ell_p^u(E_i)$, $i=1,\dots, m+1$
\begin{align*}
	&\int_{[0, 1]^{m+1}}\left\|\sum_{j_1,\dots, j_{m+1}=1}^{\infty}r_{j_1}(t_1)\cdots r_{j_{m+1}}(t_{m+1})\varphi T\left(x_{j_1}^{(1)},\dots, x_{j_{m+1}}^{(m+1)}\right) \right\|^2dt_1\cdots dt_{m+1}\\
	&=\int_0^1\left\|\sum_{j=1}^{\infty}r_{j}(t)\varphi\left(x_{j}^{(m+1)}\right) \right\|^2dt\int_{[0, 1]^{m}}\left\|\sum_{j_1,\dots, j_{m}=1}^{\infty}r_{j_1}(t_1)\cdots r_{j_{m}}(t_{m}) T\left(x_{j_1}^{(1)},\dots, x_{j_{m}}^{(m)}\right) \right\|^2dt_1\cdots dt_{m}\\
	&= \sum_{j=1}^{\infty}\left\|\varphi\left(x_j^{(m+1)}\right) \right\|^2 \int_{[0, 1]^{m}}\left\|\sum_{j_1,\dots, j_{m}=1}^{\infty}r_{j_1}(t_1)\cdots r_{j_{m}}(t_{m}) T\left(x_{j_1}^{(1)},\dots, x_{j_{m}}^{(m)}\right) \right\|^2dt_1\cdots dt_{m}\\
	&\le \|\varphi\|^2 \left\|\left(x_j^{(m+1)}\right)_{j=1}^{\infty} \right\|_{w, 2}^2 \int_{[0, 1]^{m}}\left\|\sum_{j_1,\dots, j_{m}=1}^{\infty}r_{j_1}(t_1)\cdots r_{j_{m}}(t_{m}) T\left(x_{j_1}^{(1)},\dots, x_{j_{m}}^{(m)}\right) \right\|^2dt_1\cdots dt_{m}\\
	&< \infty.
\end{align*}
Therefore, $\varphi T \in \mathcal{L}_{al, p}^{multi, m+1}(E_1,\dots, E_{m+1})$ and
\begin{align*}
	\|\varphi T\|_{\mathcal{L}_{al, p}^{multi, m+1}} &= \left\|\left(\varphi T\right)^{\wedge} \right\|\\
	&= \sup_{\left\|\left(x_j^{(i)}\right)_{j=1}^{\infty} \right\|_{w, p} \le 1}\left\|\left(\varphi T\right)^{\wedge}\left(\left(x_j^{(1)}\right)_{j=1}^{\infty},\dots, \left(x_j^{(m+1)}\right)_{j=1}^{\infty}\right)\right\|_{Rad\left(F; \mathbb{N}^{m+1}\right)}\\
	&= \sup_{\left\|\left(x_j^{(i)}\right)_{j=1}^{\infty} \right\|_{w, p} \le 1}\left\|\left(\varphi T \left(x_{j_1}^{(1)},\dots, x_{j_{m+1}}^{(m+1)}\right)\right)_{j_1,\dots, j_{m+1}=1}^{\infty} \right\|_{Rad\left(F; \mathbb{N}^{m+1}\right)}\\
	&\le \sup_{\left\|\left(x_j^{(i)}\right)_{j=1}^{\infty} \right\|_{w, p} \le 1} \|\varphi\| \left\|\left(x_j^{(m+1)}\right)_{j=1}^{\infty} \right\|_{w, p} \left\|\left(T \left(x_{j_1}^{(1)},\dots, x_{j_{m}}^{(m)}\right)\right)_{j_1,\dots, j_{m}=1}^{\infty} \right\|_{Rad\left(F; \mathbb{N}^m\right)}\\
	&=  \sup_{\left\|\left(x_j^{(i)}\right)_{j=1}^{\infty} \right\|_{w, p} \le 1} \|\varphi\| \left\|\left(x_j^{(m+1)}\right)_{j=1}^{\infty} \right\|_{w, p} \left\|\overset{\wedge}{T}\left(\left(x_j^{(1)}\right)_{j=1}^{\infty},\dots, \left(x_j^{(m)}\right)_{j=1}^{\infty}\right)\right\|_{Rad\left(F; \mathbb{N}^{m}\right)}\\
	&\le \sup_{\left\|\left(x_j^{(i)}\right)_{j=1}^{\infty} \right\|_{w, p} \le 1} \|\varphi\| \left\|\left(x_j^{(m+1)}\right)_{j=1}^{\infty} \right\|_{w, p} \left\|\overset{\wedge}{T} \right\|\prod_{i=1}^{m}\left\|\left(x_j^{(i)}\right)_{j=1}^{\infty} \right\|_{w, p}\\
	&= \|\varphi\| \|T\|_{\mathcal{L}_{al, p}^{multi, m+1}}.
\end{align*}

Now,  we check (CH4). Let $P \in \mathcal{P}_{al, p}^{multi, m}(E; F)$ and $\varphi \in E'$. As we did  in (CH2), to see that $\varphi P \in \mathcal{P}_{al, p}^{multi, m+1}(E; F)$, we just need to show that $(\varphi P)^{\vee} \in \mathcal{L}_{al, p}^{multi, m}\left(E^m; F\right)$. Note that, 
\begin{equation*}
	(\varphi P)^{\vee}\left(x_{j_1}^{(1)},\dots, x_{j_{m+1}}^{(m+1)} \right) = \displaystyle\frac{\varphi\left(x_{j_1}^{(1)}\right)\check{P}\left(x_{j_2}^{(2)} , \dots, x_{j_{m+1}}^{(m+1)}\right) + \cdots + \varphi\left(x_{j_{m+1}}^{(m+1)}\right)\check{P}\left(x_{j_1}^{(1)}, \dots, x_{j_m}^{(m)}\right)}{(m+1)!}.
\end{equation*}
Since $P \in \mathcal{P}_{al, p}^{multi, m}(E; F)$, we have $\check{P} \in \mathcal{L}_{al, p}^{multi, m}\left(E^m; F\right)$. Thus, for every $\left(x_j^{(i)}\right)_{j=1}^{\infty} \in \ell_p^u(E)$, $i=1,\dots, m+1$
\begin{equation*}
	\left(\check{P}\left(x_{j_2}^{(2)} , \dots, x_{j_{m+1}}^{(m+1)}\right)\right)_{j_2,\dots,j_{m+1}=1}^{\infty},\dots,\left(\check{P}\left(x_{j_1}^{(1)},\dots, x_{j_m}^{(m)} \right) \right)_{j_1,\dots, j_m=1}^{\infty} \in Rad\left(F; \mathbb{N}^m\right).
\end{equation*}
By (CH3), we have 
\begin{equation*}
	\left(\varphi\left(x_{j_k}^{(k)} \right) \check{P}\left(x_{j_1}^{(1)},\dots, x_{j_{k-1}}^{(k-1)}, x_{j_{k+1}}^{(k+1)},\dots, x_{j_{m+1}}^{(m+1)} \right) \right)_{j_1,\dots, j_{m+1}=1}^{\infty} \in Rad\left(F; \mathbb{N}^{m+1}\right).
\end{equation*}
Then,
\begin{align*}
	&\left((\varphi P)^{\vee}\left(x_{j_1}^{(1)},\dots, x_{j_{m+1}}^{(m+1)} \right)\right)_{j_1,\dots, j_{m+1}=1}^{\infty}\\ 
	&= \left(\displaystyle\frac{\varphi\left(x_{j_1}^{(1)}\right)\check{P}\left(x_{j_2}^{(2)} , \dots, x_{j_{m+1}}^{(m+1)}\right) + \cdots + \varphi\left(x_{j_{m+1}}^{(m+1)}\right)\check{P}\left(x_{j_1}^{(1)}, \dots, x_{j_m}^{(m)}\right)}{(n+1)!}  \right)_{j_1,\dots, j_{m+1}=1}^{\infty}\\
\end{align*}
belongs to $Rad\left(F; \mathbb{N}^{n+1}\right)$  for every $\left(x_j^{(i)}\right)_{j=1}^{\infty} \in \ell_p^u(E)$, $i=1,\dots, m+1$. Therefore,
\begin{equation*}
	(\varphi P)^{\vee} \in \mathcal{L}_{al, p}^{multi, m+1}\left(E^m; F\right).
\end{equation*}

The condition (CH5) is immediate  by the definition of multiple almost $p$-summing $m$-homogeneous polynomials. 

Since $\beta_1 = \beta_2 = \beta_3 = 1$, by \cite[remark $3.3$]{PR14}, the sequence is coherent and compatible with $\prod_{al, p}$. 
\end{proof}

\section{Multi-type and multi-cotype of multilinear operators}\label{PGS}
We begin by proposing a multi-type and multi-cotype concept of multilinear operators.

\begin{definition}
	Let $m, n_1, \dots, n_m \in \mathbb{N}$ and $0 < p_1,\dots, p_m, q < \infty$. We say that $T \in \mathcal{L}(E_1,\dots, E_m; F)$ has:
	\begin{description}
		\item[(a)] {\it Multi-type} $(p_1,\dots, p_m)$ if there is a constant $C > 0$ such that
		\begin{equation}\label{MT}
			\left\|\left(T(x_{j_1}^{(1)},\dots, x_{j_m}^{(m)})\right)_{j_1,\dots, j_m = 1}^{n_1,\dots, n_m} \right\|_{Rad(F; \mathbb{N}^m)} \le C \prod_{i=1}^{m}\left\|\left(x_j^{(i)}\right)_{j=1}^{n_i} \right\|_{p_i}
		\end{equation}
	for any choice of  finitely many vectors $(x_{j}^{(1)},\dots, x_{j}^{(m)}) \in E_1\times \cdots \times E_m$.

	\item[(b)] {\it Multi-cotype} $q$ if there is a constant $C > 0$ such that
	\begin{equation}\label{MCT}
			\left\|\left(T(x_{j_1}^{(1)},\dots, x_{j_m}^{(m)})\right)_{j_1,\dots, j_m = 1}^{n_1,\dots, n_m} \right\|_{q} \le C \prod_{i=1}^{m}\left\|\left(x_j^{(i)}\right)_{j=1}^{n_i} \right\|_{Rad(E_i)}
	\end{equation}
for any choice of  finitely many vectors $(x_{j}^{(1)},\dots, x_{j}^{(m)}) \in E_1\times \cdots \times E_m$.
	\end{description}
\end{definition}

We denote the linear space of all $m$-linear operators with multi-type $(p_1,\dots, p_m)$ and with multi-cotype $q$ from $E_1\times \cdots \times E_m$ to $F$ by $\tau_{p_1,\dots, p_m}^{multi}(E_1,\dots, E_m; F)$ and $\mathcal{C}_q^{m, multi}(E_1,\dots,$ $E_m; F)$, respectivaly. If  $p_1=\cdots = p_m = p$,  we just write $\tau_{p}^{multi}(E_1,\dots, E_m; F)$.  
It is not difficult to see that the infimum of the constants in \eqref{MT} and \eqref{MCT},  denoted by $\|T\|_{\tau_{p_1,\dots, p_m}^{multi}}$ and $\|T\|_{\mathcal{C}_q^{m, multi}}$, are a norm in $\tau_{p_1,\dots, p_m}^{multi}(E_1,\dots, E_m; F )$ and $\mathcal{C}_q^{m, multi}(E_1,\dots, E_m; F )$, respectively.

From now on, we consider an integer $m \geq 2$  and $p_1,\dots, p_m, q$ such that
\begin{equation*}
	\frac{1}{2} \le \frac{1}{p_1}+\cdots + \frac{1}{p_m} \le 1 \text{ and  } q \ge \frac{2}{m}.
\end{equation*}

\begin{theorem}\label{MTSeq}
	Let $T \in \mathcal{L}(E_1,\dots, E_m; F)$, the statements  are equivalents:
	\begin{description}
		\item[(a)] $T$ has multi-type $(p_1,\dots, p_m)$;
		\item[(b)] $\left(T\left(x_{j_1}^{(1)},\dots, x_{j_m}^{(m)}\right)\right)_{j_1,\dots, j_m=1}^{\infty} \in Rad(F; \mathbb{N}^m)$ whenever $\left(x_j^{(i)}\right)_{j=1}^{\infty} \in \ell_{p_i}(E_i)$, $i=1,\dots, m$;
		\item[(c)] The map
		\begin{equation*}
			\hat{T}\colon \ell_{p_1}(E_1)\times \cdots \times \ell_{p_m}(E_m) \rightarrow Rad\left(F; \mathbb{N}^m\right)
		\end{equation*}
	given by
	\begin{equation*}
		\hat{T}\left(\left(x_j^{(1)}\right)_{j=1}^{\infty},\dots, \left(x_j^{(m)}\right)_{j=1}^{\infty}\right) = \left(T\left(x_{j_1}^{(1)},\dots, x_{j_m}^{(m)}\right)\right)_{j_1,\dots, j_m=1}^{\infty}
	\end{equation*}
is {\bf a } well-defined continuous $m$-linear operator.
	\end{description} 

\end{theorem}

\begin{proof}
It is immediate from Proposition \ref{MPC} that  $(a)$ implies $(b)$. It also is immediate to see  that $(c)$ implies  $(a)$,  and the prove of  $(b)$ implies $(c)$ is analogous to the proof  of Proposition \ref{CONT}.

\end{proof}

Following standard  computations, we have that $\|\hat{T}\| = \|T\|_{\tau_{p_1,\dots, p_m}^{multi}}$.

The proof of the next result is similar to the proof of  Theorem \ref{MTSeq}.  

\begin{theorem}\label{MCTSeq}
	Let $T \in \mathcal{L}(E_1,\dots, E_m; F)$, the statements are equivalents:
	\begin{description}
		\item[(a)] $T$ has multi-cotype $q$;
		\item[(b)] $\left(T\left(x_{j_1}^{(1)},\dots, x_{j_m}^{(m)}\right)\right)_{j_1,\dots, j_m=1}^{\infty} \in \ell_q(F; \mathbb{N}^m)$ whenever $\left(x_j^{(i)}\right)_{j=1}^{\infty} \in Rad(E_i)$, $i=1,\dots, m$;
		\item[(c)] The map
		\begin{equation*}
			\hat{T}\colon Rad(E_1)\times \cdots \times Rad(E_m) \rightarrow \ell_q\left(F; \mathbb{N}^m\right)
		\end{equation*}
		given by
		\begin{equation*}
			\hat{T}\left(\left(x_j^{(1)}\right)_{j=1}^{\infty},\dots, \left(x_j^{(m)}\right)_{j=1}^{\infty}\right) = \left(T\left(x_{j_1}^{(1)},\dots, x_{j_m}^{(m)}\right)\right)_{j_1,\dots, j_m=1}^{\infty}
		\end{equation*}
		is a well-defined continuous $m$-linear operator.
	\end{description}
\end{theorem}

Following standard computations, we have that $\|\hat{T}\| = \|T\|_{\mathcal{C}_q^{m, multi}}$.

\begin{proposition}\label{MasmT}
	Let  $E_1,\dots, E_m, F$ be Banach spaces. Then,
	\begin{equation*}
		\mathcal{L}_{al, (p_1,\dots, p_m)}^{multi}(E_1,\dots, E_m; F) \overset{1}{\hookrightarrow} \tau_{p_1,\dots, p_m}^{multi}(E_1,\dots, E_m; F).
	\end{equation*}
\end{proposition}

\begin{proof}
	Let $T \in \mathcal{L}_{al, (p_1,\dots, p_m)}^{multi}(E_1,\dots, E_m; F)$.  As $\ell_{p_i}(E_i) \subset \ell_{p_i}^u(E_i)$, for every $\left(x_j^{(i)}\right)_{j=1}^{\infty} \in \ell_{p_i}(E_i)$, 
	we have 
	\begin{equation*}
		\left(T\left(x_{j_1}^{(1)},\dots, x_{j_m}^{(m)}\right)\right)_{j_1,\dots, j_m=1}^{\infty} \in Rad\left(F; \mathbb{N}^m\right).
	\end{equation*}
So, $T \in \tau_{p_1,\dots, p_m}^{multi}(E_1,\dots, E_m; F)$ and
\begin{align*}
	\|T\|_{\tau_{p_1,\dots, p_m}^{multi}} = \left\| \hat{T}\right\| &= \sup_{\left\|\left(x_j^{(i)}\right)_{j=1}^{\infty} \right\|_{p_i}\le 1}\left\|\hat{T}\left(\left(x_j^{(1)}\right)_{j=1}^{\infty},\dots, \left(x_j^{(m)}\right)_{j=1}^{\infty}\right) \right\|_{Rad\left(F; \mathbb{N}^m\right)}\\
	&\le \sup_{\left\|\left(x_j^{(i)}\right)_{j=1}^{\infty} \right\|_{w, p_i}\le 1}\left\|\overset{\wedge}{T}\left(\left(x_j^{(1)}\right)_{j=1}^{\infty},\dots, \left(x_j^{(m)}\right)_{j=1}^{\infty}\right) \right\|_{Rad\left(F; \mathbb{N}^m\right)}\\
	&= \left\|\overset{\wedge}{T} \right\|\\
	&= \|T\|_{\mathcal{L}_{al, (p_1,\dots, p_m)}^{multi}}.
\end{align*}
\end{proof}

In other words, every multiple almost $(p_1,\dots, p_m)$-summing  operator has multi-type $(p_1,\dots, p_m)$.

According to \cite{BC16}, we can prove the following result. 

\begin{proposition}
	Let $T \in \mathcal{L}(E_1,\dots, E_m; F)$. If $T$ has some proper multi-cotype $q$, then $T$ has some proper cotype $q$, according to \cite[Definition $3.1$]{BC16}.
\end{proposition}
\begin{proof}
 Suppose that $T \in \mathcal{L}(E_1,\dots, E_m; F)$ has some proper multi-cotype $q$, thus by Theorem \ref{MCTSeq}
\begin{equation*}
	\left(T\left(x_{j_1}^{(1)},\dots, x_{j_m}^{(m)}\right)\right)_{j_1,\dots, j_m=1}^{\infty} \in \ell_q\left(F; \mathbb{N}^m\right),
\end{equation*}
for every $\left(x_j^{(i)}\right)_{j=1}^{\infty} \in Rad(E_i)$, $i=1,\dots, m$. Since $\ell_q(\cdot; \mathbb{N}^m)$ is sequentialy compatible with $\ell_q(\cdot)$, we have
\begin{align*}
	\left(T\left(x_j^{(1)},\dots, x_j^{(m)}\right)\right)_{j=1}^{\infty} \in \ell_q(F) {\bf } 
\end{align*}
for every $\left(x_j^{(i)}\right)_{j=1}^{\infty} Rad(E_i)$. Therefore, by \cite[Theorem $3.6$]{BC16}, $T$ has cotype q.
\end{proof}

The Proposition  \ref{MasmT} ensure that the operator in the Example	\ref{EX1} has multi-type $(p_1,\dots, p_{m+1})$.

\begin{theorem}
 $\left(\tau_{p_1,\dots, p_m}^{multi}, \|\cdot\|_{\tau_{p_1,\dots, p_m}^{multi}} \right)$ is a Banach ideal and $\left(\mathcal{C}_q^{m, multi}, \|\cdot\|_{\mathcal{C}_q^{m, multi}} \right)$ is a $q$-Banach ideal  of $m$-linear operators.
\end{theorem}

\begin{proof}
	Let $\ell_{p_i}(\cdot), Rad(\cdot)$ be linearly stable sequence classes, and let 
	$\ell_{p_i}\left(\cdot; \mathbb{N}^{\infty}\right)$ and   $Rad\left(\cdot; \mathbb{N}^m\right)$ be linearly stable $m$-sequence classes. We already know that  $\ell_{p_1} \cdots \ell_{p_m} = \ell_{p_1}^u(\mathbb{K})\cdots \ell_{p_m}^u(\mathbb{K})$ $\overset{multi, 1}{\hookrightarrow} Rad\left(\mathbb{K}; \mathbb{N}^m\right)$.  It  is also not difficult to see that	
\begin{equation*}
	Rad(\mathbb{K})\overset{m}{\cdots} Rad(\mathbb{K}) = \ell_2 \overset{m}{\cdots} \ell_2 \overset{multi, 1}{\hookrightarrow} \ell_q(\mathbb{K};\mathbb{N}^m) {\bf }
\end{equation*}
because, $q \ge \frac{2}{m}$. Since $Rad(\cdot), \ell_2(\cdot), \ell_{p_i}(\cdot)$ are linearly stable sequence classes and $Rad\left(\cdot; \mathbb{N}^m\right)$ and $\ell_2\left(\cdot; \mathbb{N}^m\right)$ are linearly stable $m$-sequence classes, it follows from \cite[Theorem $3.5$]{RS19} that $\left(\tau_{p_1,\dots, p_m}^{multi}, \|\cdot\|_{\tau_{p_1,\dots, p_m}^{multi}} \right)$ is a Banach ideal and $\left(\mathcal{C}_q^{m, multi}, \|\cdot\|_{\mathcal{C}_q^{m, multi}} \right)$ is a $q$-Banach ideal  of $m$-linear operators.
\end{proof}

Following the same idea as in  Theorem \ref{CCMAS}, we can obtain the following result.

\begin{theorem}
	The pairs
	\begin{align*}
		&\left(\left(\tau_{p}^{m, multi}, \|\cdot\|_{\tau_{p}^{m, multi}}\right), \left(\mathcal{P}_{\tau_{p}^{m, multi}}, \|\cdot\|_{\mathcal{P}_{\tau_{p}^{m, multi}}}\right)\right)_{m=1}^{\infty}\\ &\text{ and } \left(\left(\mathcal{C}_q^{m, multi}, \|\cdot\|_{\mathcal{C}_q^{m, multi}}\right), \left(\mathcal{P}_{\mathcal{C}_q^{m, multi}}, \|\cdot\|_{\mathcal{P}_{\mathcal{C}_q^{m, multi}}}\right)\right)_{m=1}^{\infty}
	\end{align*}
are coherent and compatible with $\tau_p$ and $\mathcal{C}_q$, respectively.
\end{theorem}


\begin{thebibliography}{C}

\bibitem{BJ} Braunss H.-A., Junek, H.:{\bf Ideals of polynomials and multilinear mappings}.
Unpublished notes.

\bibitem{BG04} Bombal F., Garcia D., Villanueva I.: {\bf Multilinear extensions of Grothendieck’s
theorem}. The Quarterly Journal of Mathematics. {\bf 55}(4), 441–450, (2004).

\bibitem{B00} Botelho G.: {\bf Almost summing polynomials}. Math. Nachr. {\bf 211}, 25–36, (2000).

\bibitem{B05} Botelho G.: {\bf Ideals of polynomials generated by weakly compact operators}.
Note di Matematica {\bf 25}, 69–102, (2005).

\bibitem{BC16} Botelho G., Campos, J.: {\bf Type and cotype of multilinear operators}. Revista Matemática Complutense, {\bf 29}, 659-676, (2016).

\bibitem{BC17} Botelho G., Campos J.: {\bf On the transformation of vector-valued sequences by linear and multilinear operators}. Monatshefte fur Mathematik. {\bf 183}, 415--435, (2017).

\bibitem{BBJ01}  Botelho G., Braunss H.-A.,  Junek H.: {\bf Almost p-summing polynomials and multilinear
mappings}, Arch. Math., {\bf 76}, 109–118, (2001).

\bibitem{DJT95} Diestel J., Jarchow H., Tonge A.: {\bf Absolutely Summing Operators}. Cambridge University Press, Cambridge, (1995).

\bibitem{FG03} Floret K., Garcia D. : {\bf On ideals of polynomials and multilinear mappings
between Banach spaces}. Archiv der Mathematik. {\bf 81}, 300-308, (2003).

\bibitem{M03} Matos M.C.: {\bf Fully absolutely summing and Hilbert-Schmidt multilinear mapping}.
Collectanea Mathematica. {\bf 54}, 111–136, (2003).

\bibitem{P06} Pellegrino  D.,  Souza M. L. V.: {\bf Fully and strongly almost summing
multilinear mappings}, Rocky Mt. J. Math. {\bf 36}, 683–698, (2006).

\bibitem{P03} Pellegrino D.: {\bf Strongly almost summing holomorphic mappings}. J. Math. Anal Appl., {\bf 287}, 246–254, (2003).

\bibitem{P04} Pellegrino D.: {\bf Almost summing mappings}. Arch. Math., {\bf 82}, 68–80, (2004).

\bibitem{PR12} Pellegrino, D., Ribeiro, J.: {\bf On almost summing polynomials and multilinear mappings}. Linear and Multilinear Algebra. {\bf 60}, 397--413, (2012).

\bibitem{PR14} Pellegrino D., Ribeiro J.:{\bf On multi-ideals and polynomial ideals of Banach spaces: a new approach to coherence and compatibility}. Monatshefte fur Mathematik. {\bf 173}, 379-415, (2014).

\bibitem{P67} Pietsch A.: {\bf Absolutely p-summierende Abbildungen in normieten Raumen}.
Stud. Math. {\bf 27}, 333–353, (1967).

\bibitem{P83} Pietsch A.: {\bf Ideals of multilinear functionals}. In: Proceedings of the Second
International Conference on Operator Algebras, Ideals and their Applications
in Theoretical Physics. Teubner-Texte, Leipzig, pp. 185-199, (1983).

\bibitem{P12} Popa,D.: {\bf Multiple Rademacher means and their applications}. J. Math. Anal. Appl. {\bf 386} (2), 699–708, (2012).

\bibitem{P14} Popa, D.: {\bf Almost summing and multiple almost summing multilinear
operators on $\ell_p$ spaces}. Archiv der Mathematik, {\bf 103}, 291-300, (2014).

\bibitem{RS19} Ribeiro J., Santos F.: {\bf Generalized Multiple Summing Multilinear
	Operators on Banach Spaces}.  Mediterranean Journal of Mathematics, {\bf 16}, p. 108, (2019).











%
%
%
%
%
%
%
%
%
%
%
%
%
%
%
%
%
%
%
%
%
%
%
%
%
%
%
%
%
%
%
%
%
%
%
%
%
%
%
%
%
%
%
%
%
%
%
%
%
%
%
%
%
%
%
%


\end{thebibliography}
\end{document}